\documentclass{amsart}

\usepackage{color}
\usepackage{graphicx} 
\usepackage{latexsym} 
\usepackage{hyperref}%\pagestyle{empty} \usepackage{amsfonts} \usepackage[all]{xy} %\usepackage[notref,notcite]{showkeys}
\usepackage{amssymb, mathrsfs, amsfonts, amsmath} \usepackage{amsbsy} \usepackage{amsfonts}
\usepackage{ textcomp}
\usepackage{palatino}

\newtheorem{corollary}{Corollary} 
\newtheorem*{theorem*}{Theorem} 
\newtheorem{definition}{Definition} 
 
\newtheorem{proposition}{Proposition} 
\newtheorem{remark}{Remark}
 
\newtheorem{theorem}{Theorem} 
\newtheorem{example}{Example} 
\numberwithin{equation}{section} 

\begin{document} 
	
\title[The generalized Pohoz{a}ev-Schoen identity]{The generalized Poho{z}aev-Schoen identity and some geometric applications} 

\author{Ezequiel Barbosa}
\address{Instituto de Ci\^{e}ncias Exatas-Universidade Federal de Minas Gerais\\ 30161-970-Belo Horizonte-MG-BR} 
\email{ezequiel@mat.ufmg.br}

\author{Levi Lopes de Lima}
\address{Universidade Federal do Cear\'a,
		Departamento de Matemática, Campus do Pici, R. Humberto Monte, s/n, 60455-760,
		Fortaleza/CE, Brazil.}
\email{levi@mat.ufc.br}

\author{Allan Freitas} 
\address{Instituto de Ci\^{e}ncias Exatas-Universidade Federal de Minas Gerais\\ 30161-970-Belo Horizonte-MG-BR}
\email{allangeorge@ufmg.br} 
\thanks{ The first and second authors were partially supported by CNPq/Brazil and  the third one was supported by Fapemig/MG}

\keywords{generalized Poho{z}aev-Schoen identity, $V$-static manifolds, Ricci solitons} 

\subjclass[2000]{Primary 53C25, 53C20,
53C21; Secondary 53C65} %\urladdr{http://www.mat.ufc.br/pgmat} \date{April 05, 2012}

\begin{abstract} 
In this note we show how a generalized Pohozaev-Schoen identity due to Gover and Orsted \cite{GO} can be used to obtain some rigidity results for $V$-static manifolds and generalized solitons. We also obtain an Alexandrov type result for certain  hypersurfaces in Einstein manifolds.
\end{abstract}

\maketitle

\section{Introduction}\label{intro}

In recent years, Riemannian metrics satisfying  certain systems of partial differential equations involving a function (and/or a vector field) defined on the underlying manifold have been  extensively studied. Examples include geometric structures related to the class of static metrics appearing in General Relativity, such as the critical metrics  introduced in \cite{MT1}, and the generalized Ricci solitons discussed in \cite{PRRS}. The dominant theme here is to obtain rigidity results making sure that the given metric necessarily belongs to a class of ``trivial'' solutions of the corresponding system.

The purpose of this note is to point out that some  results in this area can be quickly derived from a Pohozaev-type integral identity due to Schoen \cite{S}. In fact, we consider here  a generalization of Schoen's identity  due to Gover and Orsted \cite{GO}. This remarkable formula, which is a natural outgrowth of E. Noether's brilliant insight that infinite dimensional symmetry groups  lead to locally conserved quantities, actually 
encompasses a broad class of integral identities previously established in the literature. As illustrated below through examples, the identity can be used to access many interesting results on the topics mentioned above. Applications include rigidity statements for $V$-static manifolds in terms of the boundary behavior of the metric and certain classification results for generalized Ricci solitons under suitable  assumptions on the underlying geometry. 

This paper is organized as follows. In Section \ref{pohoschid}
we describe the generalized Pohozaev-Schoen integral identity due to Gover-Orsted \cite{GO} and present a list of locally conserved tensor fields to which it applies.     
In Section \ref{vstatic} we apply the classical special case due to Schoen to discuss several results for $V$-static manifolds. Examples of such results include a sharp upper bound for the area of the boundary of a  $V$-static $3$-manifold (Theorem \ref{main-thmD}), which extends a classical estimate due to Boucher-Gibbons-Horowitz \cite{BGH} for static manifolds, and a characterization in terms of the mean curvature of $V$-static manifolds in case the induced metric on the boundary is isometric to a geodesic sphere in a simply connected space form (Theorem \ref{main-thmBBB}), which extends a result due to Miao-Tam \cite{MT1}. When the $V$-static manifold is spin and has non-positive scalar curvature, we show in Section \ref{vstatspin} how ideas in \cite{HMRa} can be used to generalize this latter result under  the much weaker assumption that the induced metric on the boundary dominates the round metric on the geodesic sphere.   Finally, in Section \ref{applpsvstac} we discuss applications of the generalized Pohozaev-Schoen identity to solitons. In particular, we provide alternative proofs to results appearing in \cite{BBR} and \cite{GWX} (Theorems \ref{gensolbbr} and \ref{gensolgwx}). We also include in this section an Alexandrov-type result for a certain class of hypersurfaces in Einstein manifolds (Theorem \ref{alexth}).

\section{The generalized Pohozaev-Schoen integral identity}\label{pohoschid}

In this section we discuss the generalized  Poho{z}aev-Schoen identity presented in \cite{GO}.  
An earlier manifestation of this identity appears in an old paper by Pohozaev \cite{Po}, where it is proved that Dirichlet solutions of certain semilinear elliptic equation in domains of Euclidean space satisfy an integral identity. 
As another installment of this same principle,
Bourguignon and Ezin \cite{BE} proved that, in the  presence of a conformal vector field, the scalar curvature of a closed Riemannian manifold satisfies a similar  integral identity.
It turns out that 
both identities above are special cases of a more general identity due to Schoen \cite{S}. In the following we will always denote by $\stackrel{\circ}{B}$
the trace free part of a symmetric $2$-tensor
$B$. Moreover, unless otherwise stated, all manifolds are assumed to be connected and oriented.

\begin{theorem}[Pohoz{a}ev-Schoen integral identity \cite{S}]\label{schoenid}
Let $(M^{n},g)$ be a compact Riemannian manifold with boundary $\Sigma$ and $X$ a vector field on $M$. Then there holds 
\begin{equation}\label{schoenid2}
\int_{M}X(R_g)dM=-\frac{n}{n-2}\int_{M} \left\langle {\stackrel{\circ}{{\rm Ric}_g}},\mathcal{L}_{X}g\right\rangle dM+\frac{2n}{n-2}\int_{\Sigma}{\stackrel{\circ}{{\rm Ric}_g}}(X,\nu)d\Sigma,
\end{equation}
where  ${\rm Ric}_g$ is the Ricci tensor, $R_g$ is the scalar curvature, ${\mathcal L}$ is  Lie derivative and $\nu$ is the outward unit normal vector field.
\end{theorem}

Given the relevance of (\ref{schoenid2}) in connection with several deep questions in Geometric Analysis, it is natural to seek for a derivation of this identity from basic principles. This has been accomplished in \cite{GO}. Realizing that the proof of (\ref{schoenid2}) is a consequence of the contracted Bianchi identity and integration by parts,  these authors were able to deduce a rather general version of (\ref{schoenid2}).
In this regard, the next definition isolates the key concept interpolating between  the infinitesimal symmetries defined by vector fields and the corresponding integral identity.

\begin{definition}\label{loccons}
A symmetric 2-tensor $B$ on a Riemannian manifold is said to be {\em locally conserved} if it is divergence free, i.e. $\nabla^{i}B_{ij}=0$. \label{locallyconserverd}
\end{definition}

We now display examples of tensor fields fitting into this definition. 

\begin{example}\label{einsteinten}
	{\rm The contracted Bianchi identity in Riemannian (or Lorentzian) geometry says that the Einstein tensor 
		\[
		E_g={\rm Ric}_g-\frac{R_g}{2}g
		\]
		is locally conserved. This plays a central role in General Relativity and is of course at the core of the classical work by E. Noether on the relationship between symmetries and conservation laws. Notice that $\stackrel{\circ}{E}_g\,=\,\stackrel{\circ}{{\rm Ric}}_g$.}
	\end{example}

\begin{example}\label{schouten}
{\rm Set $Q_g=P_g-Jg$, where 
\[
P_g=\frac{1}{n-2}\left({\rm Ric}_g-\frac{R_g}{2(n-1)}g\right),
\]
is the Schouten tensor and 
$J=g^{ij}P_{ij}$ is its trace. It follows once again from the contracted Bianchi identity that $Q_g$ is locally conserved.}
\end{example}

\begin{example}\label{hyper}
{\rm Let $(M,g)$ be a hypersurface of Euclidean space $\mathbb{R}^{n+1}$, $II$ its second fundamental form and $H=g^{ij}II_{ij}$ the mean curvature.  It follows that $II$ satisfies the contracted Codazzi equation
\[
\nabla^{i}II_{ij}=\nabla_{j}H,
\]
and therefore the symmetric 2-tensor $A=II-Hg$ is locally conserved. This result holds more generally if $\mathbb R^{n+1}$ is replaced by any Einstein manifold;  see the proof of Theorem \ref{alexth} below.}
\end{example}

\begin{example}
	\label{lovelock} {\rm (Lovelock tensors)  Let $(M^n,g)$ be a Riemannian manifold and $dM$ the associated volume element. Given a vector field $X\in\mathcal X(M)$ we have 
\[
d({\bf i}_XdM)=d({\bf i}_X dM)+{\bf i}_X d(dM)=\mathcal L_XdM=({\rm div}_gX)dM,
\]		
where $\mathcal L$ is Lie derivative and ${\bf i}_X$ is contractoin with $X$.
Thus, the correspondence $X\leftrightarrow\omega= {\bf i}_XdM$ defines an isomorphism between $\mathcal A^1(M)$ and $\mathcal A^{n-1}(M)$, where $\mathcal A^p(M)$ denotes the space of differential $p$-forms on $M$. Obviously, $X$ is divergence free if and only if $\omega$ is closed. Similarly, the correspondence $X_1\otimes X_2\leftrightarrow {\bf i}_{X_1} dM\otimes {\bf i}_{X_2} dM$ defines an isomorphism between $\mathcal T^2(M)$, the space of covariant $2$-tensors over $M$ and $\mathcal A^{n-1}(M)\otimes_{\mathcal A^0(M)}\mathcal A^{n-1}(M)$. By restriction we obtain an isomorphism between $\mathcal S^2(M)\subset \mathcal T^2(M)$, the subspace  of symmetric $2$-tensors, and $\mathcal S^2(\mathcal A^{n-1}(M))\subset A^{n-1}(M)\otimes_{\mathcal A^0(M)}\mathcal A^{n-1}(M)$, where by definition $\eta\in \mathcal S^2(\mathcal A^{p}(M))$ if and only if 
it
is symmetric in the sense that
\[
\eta(e_{i_1}\wedge\cdots\wedge e_{i_p}\otimes e_{j_1}\wedge\cdots\wedge e_{j_p})=\eta(e_{j_1}\wedge\cdots\wedge e_{j_p}\otimes e_{i_1}\wedge\cdots\wedge e_{i_p}).
\]		 
It is easy to check that $B\in\mathcal S^2(M)$ is locally conserved if and only if the corresponding element $\widetilde B\in \mathcal S^2(\mathcal A^{n-1}(M))$ satisfies $d^\nabla\widetilde B=0$, where $d^\nabla$ is the covariant exterior differential. 	Now, for each $1\leq k\leq [(n-1)/2]$ define $\widetilde L_{2k}\in \mathcal S^2(\mathcal A^{n-1}(M))$ by
\[
\widetilde L_{2k}=R_g\wedge\stackrel{k}{\cdots}\wedge R_g\wedge g\wedge\cdots\wedge g,
\]	
where $R_g$ is the curvature tensor of $g$.
Since $d^\nabla g=0$ ($g$ is parallel) and $d^\nabla R_g=0$ (Bianchi identity)	we see that $d^\nabla\widetilde L_{2k}=0$. Thus, the corresponding symmetric two-tensor $L_{2k}$, called the {\em Lovelock tensor}, is locally conserved \cite{L}. It is easy to check that $L_2$ is a multiple of $E_g$, the Einstein tensor.}
\end{example}

We now state the generalized Pohozaev-Schoen identity considered in \cite{GO}.

\begin{theorem} \cite{GO}
\label{gpi}
Let $(M,g)$ be a compact Riemannian manifold with boundary $\Sigma$. If $B$ is a locally conserved symmetric 2-tensor and $X$ is a vector field on $M$ there holds 
\begin{equation}\label{schoenpohgen}
  \int_{M}X(b)dM=\frac{n}{2}\int_{M} \left\langle {\stackrel{\circ}{B}},\mathcal{L}_{X}g\right\rangle dM-n\int_{\Sigma}{\stackrel{\circ}{B}}(X,\nu)d\Sigma,
 \end{equation}
 where $b=g^{ij}B_{ij}$ is the trace of $B$.
\end{theorem}

For the sake of completeness we include here the simple  proof of this result.
First observe that integration by parts yields 
\[
\int_{\Sigma}B(X,\nu)d\Sigma=\displaystyle\int_{M}\nabla^{i}(B_{ij}X^{j})dM.
\]
Since $B$ is locally conserved and symmetric we find that 
\[
\nabla^{i}(B_{ij}X^{j})=B_{ij}\nabla^{i}X^{j}=\frac{1}{2}B_{ij}(\nabla^{i}X^{j}+\nabla^{j}X^{i})=\frac{1}{2}\langle B,\mathcal{L}_{X}g\rangle\,.
\]
Therefore,
\begin{eqnarray*}
 \int_{\Sigma}B(X,\nu)d\Sigma&=&\frac{1}{2}\int_{M}\langle B,\mathcal{L}_{X}g\rangle dM\\
  & = & \frac{1}{2}\int_{M}\langle{\stackrel{\circ}{B}},\mathcal{L}_{X}g\rangle dM+\frac{1}{2n}\int_{M}b\,\langle g,\mathcal{L}_{X}g\rangle dM\\
  & = & \frac{1}{2}\int_{M}\langle{\stackrel{\circ}{B}},\mathcal{L}_{X}g\rangle d\mu_{g}+\frac{1}{n}\int_{M}b\,{\rm div}_gX d\mu_{g}.
 \end{eqnarray*}
Since 
\[
 \int_Mb\,{\rm div}\,XdM=-\int_{M}X(b)dM
 +\int_{\Sigma}b\, g(X,\nu) d\Sigma,
\]
the result follows.

We now list a few useful consequences of this result. 

\begin{corollary}
If $M$ is closed  then
 \begin{equation}\label{schpohocol1}
  \int_{M}X(b)dM=\frac{n}{2}\displaystyle\int_{M} \left\langle {\stackrel{\circ}{B}},\mathcal{L}_{X}g\right\rangle dM.
 \end{equation}
\end{corollary}

\begin{corollary}
If $X$ is a conformal vector field, i.e., ${\mathcal{L}}_{X} g= fg$ for some function $f$ on $M$, then
 \begin{equation}\label{schpohocol2}
  \int_{M}X(b)dM=-n\int_{\partial M}{\stackrel{\circ}{B}}(X,\nu)d\Sigma.
 \end{equation}
In particular, $\int_{M} X(b)dM=0$ if $M$ is closed.
\end{corollary}

\begin{remark}
\label{reduc}{\rm It is easy to check that the classical Pohozaev-Schoen identity (\ref{schoenid2}) follows from (\ref{schoenpohgen}) by taking $B=E_g$ as in Example \ref{einsteinten}.}
\end{remark}

\section{Applications of the Pohozahev-Shoen identity to $V$-static manifolds}\label{vstatic}

Let $M$ be a connected, smooth $n$-manifold with (a possibly disconnected) boundary $\Sigma$. Fix a metric $\gamma$ 
on $\Sigma$ and let $\mathcal{M}_{\gamma}$ be the set of all metrics on $M$ such
that $g|_{T\Sigma}=\gamma$. Next we consider the volume functional
$
V:\mathcal{M}_{\gamma}\rightarrow \mathbb{R}.
$
Fix a constant $K$ and consider the set $\mathcal{M}^{K}_{\gamma}$ of all  metrics $g\in \mathcal{M}_{\gamma}$ such that $R_g=K$. In \cite{MT1} it is proved that a metric $g\in \mathcal{M}^{K}_{\gamma}$
with the property that the first Dirichlet eigenvalue of $(n - 1)\Delta_g + K$ is positive is a critical
point for the volume functional in $\mathcal{M}^{K}_{\gamma}$
if and only if there exists a smooth function $\lambda$ on $M$ such that
\begin{equation}\label{vstat}
-(\Delta_g \lambda)g + \nabla^2_g\lambda-\lambda {\rm Ric}_g=g\quad {\rm on}\quad {\rm int}\, M,
\end{equation}
and $\lambda|_{\Sigma}=0$, where $\Delta_g$ and $\nabla^2_g$ are the Laplacian and the Hessian operator with respect to $g$, respectively. We then say that $g$ is a $V$-static metric with potential function $\lambda$ and that $(M,g)$ is a $V$-static manifold. 
We may assume that $R_g=K=\varepsilon n(n-1)$, $\varepsilon=-1,0,1$.
We then set $\mathcal M^\varepsilon_\gamma=\mathcal M^K_\gamma$. Also, it is known that $\lambda\geq 0$ and $\Sigma=\lambda^{-1}(0)$. In fact, each connected component of $\Sigma$ is umbilical with mean curvature given by $H=-(\partial\lambda/\partial\nu)^{-1}>0$. 
In the following we always assume that $M$ is compact. For further results on the geometry of $V$-static manifolds see \cite{MT1}, \cite{MT2} and \cite{CEM}.

In \cite[Theorem 6]{MT1} it is shown that among compact domains  in  simply connected space forms, geodesic balls are the only $V$-static manifolds (in the spherical case one has to assume that the ball is contained in an open hemisphere). 
A natural question here is to ask if this characterizations remains valid in a larger class of metrics. The next result  provides a partial answer to this question.

\begin{theorem}\cite[Theorem 2.1, Theorem 4.1]{MT2}\label{MTam}
Let $(M,g)$ be a
compact $V$-static manifold which is either Einstein or simply connected and locally conformally flat. If $\Sigma$ is connected then $(M,g)$ is isometric to a geodesic ball in a simply connected space form.
\end{theorem}

Similar results in low dimensions  have been obtained in 
\cite{BDR}.
Here we use the Pohozaev-Schoen integral identity (\ref{schoenpohgen}) to characterize these examples as the limiting cases of certain geometric inequalities. The next definition provides some motivation for our results.

\begin{definition}\label{staticm}
A complete and connected Riemannian manifold $(M^{n},g)$ with a (possibly nonempty) boundary $\Sigma$ is said to be {\em static} if it there exists a  non-negative function $\lambda$ on $M$ satisfying 
\begin{equation}
-(\Delta_g \lambda)g + \nabla^2_g\lambda-\lambda {\rm Ric}_g=0\quad {\rm in}\quad {\rm int}\, M,
\label{Staticequation}
\end{equation}
and $\Sigma=\lambda^{-1}(0)$.
\end{definition}

The left hand side of (\ref{Staticequation}) happens to be  the formal $L^{2}$-adjoint of the linearization of the scalar curvature operator and as such it plays a central role in problems involving the prescription of this invariant. Also, the identity  appears in General Relativity, where it defines static solutions of Einstein field equations. 
It is easy to check that the scalar curvature $R_g$ of $g$ in (\ref{Staticequation}) is necessarily constant, so we may assume that $R_g=\varepsilon n(n-1)$, $\varepsilon=0,1,-1$. Moreover, $\Sigma=\lambda^{-1}(0)$ is a totally geodesic hypersurface.

Let $(M,g)$ be a static manifold with $\Sigma= \bigcup\limits_{i=1}^l \Sigma_i$, where $\Sigma_{i}$ are the connected components of $\Sigma$. When $\varepsilon=1$ Chru\'schiel \cite{Cr} showed that
\begin{equation}\label{crus}
\sum\limits_{i=1}^l\kappa_i\int_{\Sigma_i}\left(R_{\gamma_i}-(n-2)(n-1)\right)d\Sigma_i\geq0,
\end{equation}
where $\gamma=g|_{\Sigma}$, $\gamma_i=\gamma|_{\Sigma_i}$ and $\kappa_i$ is the restriction of $|\nabla_g \lambda|$ to ${\Sigma_i}$. Moreover, the equality implies that $M$ is a round hemisphere and, {\em a fortiori}, $l=1$. This inequality has important applications. For instance, if $n=3$ and $\Sigma$ is connected it leads to the  famous Boucher-Gibbons-Horowitz \cite{BGH} upper bound for the area $|\Sigma|$ of $\Sigma$:
\begin{equation}\label{bgh}
|\Sigma|\leq 4\pi,
\end{equation}
with the equality holding if and only if $M$ is isometric to the round hemisphere (de Sitter space). Further results on static-type metrics can be found in \cite{QY} and the references therein.
\indent\par
We now give natural extensions of these results to $V$-static manifolds.
We start with a general integral identity which is a direct consequence of (\ref{schoenpohgen}).

\begin{theorem}\label{main-ineq}
	Let $(M^n,g)$ be a compact  $V$-static manifold with boundary $\Sigma= \bigcup\limits_{i=1}^l \Sigma_i$, and $R_g=\varepsilon n(n-1)$, where $\varepsilon=-1,0,1$. Then the following integral identity holds:
	\begin{equation}\label{ineqvstat}
	\int_M\lambda|\stackrel{\circ}{\rm Ric}_g|^2dM=-\sum_i\kappa_i\int_{\Sigma_i}\left({\rm Ric}_{g}(\nu,\nu)-\varepsilon(n-1)\right)d\Sigma_i,
	\end{equation}
	where $\kappa_i$ is the restriction of $|\nabla_g \lambda|$ to ${\Sigma_i}$. In particular,
		\begin{equation}\label{EQ}
		\sum_i\kappa_i\int_{\Sigma_i}\left({\rm Ric}_g(\nu,\nu)-\varepsilon(n-1)\right)d\Sigma_i\leq 0,
		\end{equation}
		with the equality holding if and only if $(M,g)$ is isometric to a geodesic ball in a simply connected space form.
\end{theorem}

\begin{proof}
	We apply (\ref{schoenpohgen}) with  $B={E_g}$, the Einstein tensor, and $X=\nabla_g\lambda$, so that $\stackrel{\circ}{E}=\stackrel{\circ}{\rm Ric}_g$  and $\mathcal L_{\nabla_g \lambda}g=2\nabla^2_g \lambda$. Using (\ref{vstat}) we see that 
	\[
	\langle\stackrel{\circ}{\rm Ric}_g,\nabla^2_g\lambda\rangle=\lambda\langle\stackrel{\circ}{\rm Ric}_g,{\rm Ric}_g\rangle=\lambda|\stackrel{\circ}{\rm Ric}_g|^2. 
	\]
	Thus, (\ref{ineqvstat}) follows immediately by observing that  
	$\nu=-{\nabla_g \lambda}/{|\nabla_g \lambda|_g}$.
	Finally, if the equality holds in (\ref{EQ}) then from (\ref{ineqvstat}) we conclude that $(M,g)$ is Einstein and the rigidity statement follows from   
	Theorem \ref{MTam}.
\end{proof}

\begin{corollary}\label{main-thmAAA}
	Let $(M,g)$ be as in the theorem and assume that 
	\[
	{\rm Ric}_g(\nu,\nu)\geq\varepsilon(n-1)
	\]
	along $\Sigma$.
	Then $(M,g)$ is isometric to a geodesic ball in
a simply connected space form.
\end{corollary}

The next result is an immediate consequence of (\ref{ineqvstat}) and Gauss equation.

\begin{theorem}\label{main-thmB}
Let $(M^n,g)$ be a compact $V$-static manifold  with boundary $\Sigma= \bigcup\limits_{i=1}^l \Sigma_i$, and $R_g=\varepsilon n(n-1)$, where $\varepsilon=-1,0,1$. Then the following integral identity holds:
\begin{equation}\label{statiid}
\int_M|\stackrel{\circ}{\rm Ric}_g|^2dM=\frac{1}{2}\sum\limits_{i=1}^l\kappa_i\int_{\Sigma_i}\left(R_{\gamma_i}-\varepsilon(n-2)(n-1)-\frac{n-2}{n-1}H_i^2\right)d{\Sigma_i},
\end{equation}
where $g_i=g|_{\Sigma_i}$ and  $H_i$ is the mean curvature of $\Sigma_i$. In particular, 
\begin{equation}\label{statineq}
\sum\limits_{i=1}^l\kappa_i\int_{\Sigma_i}\left(R_{\gamma_i}-\varepsilon(n-2)(n-1)-\frac{n-2}{n-1}H_i^2\right)d{\Sigma_i}\geq0,
\end{equation}
with the equality holding if and only if $(M,g)$ is isometric to a geodesic ball in a  simply connected space form. 
\end{theorem}

\begin{proof}
Gauss equation applied to the totally umbilical embedding $\Sigma_i\subset M$ says that 
\[
2{\rm Ric}_g(\nu,\nu)+R_{\gamma_i}= \varepsilon n(n-1)+\frac{n-2}{n-1}H_i^2,
\]
which can be rewritten as 
\[
2\left({\rm Ric}_g(\nu,\nu)-\epsilon(n-1)\right)=\varepsilon(n-2)(n-1)+\frac{n-2}{n-1}H_i^2-R_{g_i}
\]
The result follows.
\end{proof}

We now present a lower bound for the volume of $V$-static manifolds with $\varepsilon\geq 0$ and a connected boundary.

\begin{theorem}\label{volumebound}
	Let $(M^n,g)$ be a compact $V$-static manifold with  a connected boundary $\Sigma$ and $R_g=\varepsilon n(n-1)$, where $\varepsilon=0,1$.  Then $\int_{\Sigma}R_{\gamma}d\Sigma>0$ and the following lower bound for the volume $|M|$ of $M$ holds:
	\begin{equation}
	\label{volumebound1}
	|M|\geq \frac{n-1}{nH}\left(
	\varepsilon(n-1)(n-2)+\frac{n-2}{n-1}H^2\right)^{-1}\int_{\Sigma}R_{\gamma}d\Sigma.
	\end{equation}
	\end{theorem}
	
	\begin{proof}
From (\ref{statineq})	we have
\begin{equation}\label{eqeq01}
\int_{\Sigma}R_{\gamma}d\Sigma\geq \left(
\varepsilon(n-1)(n-2)+\frac{n-2}{n-1}H^2\right)|\Sigma|.
\end{equation}
On the other hand, taking trace of (\ref{vstat}) and integrating over $M$ we get
\[
(n-1)\int_M\Delta_g\lambda dM+\varepsilon n(n-1)\int_M\lambda dM+n|M|=0.
\]
But
\[
\int_M\Delta_g\lambda dM=\int_{\Sigma}\frac{\partial \lambda}{\partial \nu}d\Sigma=-\frac{|\Sigma|}{H},
\]
so we obtain
\begin{equation}\label{eqeq02}
|\Sigma|\geq\frac{H}{n-1}\left(n|M|+\varepsilon n(n-1)\int_M\lambda dM\right)\geq \frac{n}{n-1}H|M|,
\end{equation}
and the result follows by combining (\ref{eqeq01}) and (\ref{eqeq02}).
\end{proof}

\begin{remark}
	\label{cemmann}{\rm If $\varepsilon=0$ this theorem corresponds to \cite[Proposition 2.5 (b)]{CEM}. In this case,  the equality implies that $(M,g)$ is isometric to a geodesic ball in a simply connected space form. Clearly, if $\varepsilon=1$ then the inequality (\ref{volumebound1}) is always strict.} 
\end{remark}

As a consequence of  Theorem \ref{main-thmB} we obtain  a topological classification for the (connected) boundary of a positive $V$-static $3$-manifold.

\begin{theorem}\label{main-thmC}
Let $(M^3,g)$ be a compact $V$-static manifold with  a connected boundary $\Sigma$ and $R_g=6\varepsilon$, where $\varepsilon=-1,0,1$. If $\varepsilon=-1$ assume that $H>2$. Then 
$\Sigma$ is diffeormophic to the $2$-sphere.
\end{theorem}

\begin{proof}
It follows from Theorem \ref{main-thmB} and Gauss-Bonnet formula that
\begin{equation}\label{bghgen1}
4\pi\chi(\Sigma)\geq\left( 2\varepsilon+\frac{1}{2}H^2\right)|\Sigma|,
\end{equation}
where $\chi(\Sigma)$ is the Euler characteristic of $\Sigma$.
               Thus, $\chi(\Sigma)>0$.
\end{proof}

It is worthwhile to state the following extension of (\ref{bgh}), which is 
an immediate consequence of the proof above.

\begin{theorem}\label{main-thmD}
Let $(M^3,g)$ be a compact $V$-static manifold with a connected boundary $\Sigma$ and $R_g=6\varepsilon$, where $\varepsilon=-1,0,1$.
If $\varepsilon=-1$ assume that $H> 2$. Then
\[
|\Sigma|\leq 4\pi\left(\varepsilon+\frac{1}{4}H^2\right)^{-1}.
\]
Moreover, equality  holds if and only if $(M^3,g)$ is isometric to a geodesic ball in a simply connected space form.
\end{theorem}

Using Theorem \ref{main-thmB} and Chern-Gauss-Bonnet formula for closed 4-manifolds we obtain the following result.

\begin{theorem}\label{main-thmE}
Let $(M^5,g)$ be a compact $V$-static manifold with connected boundary $\Sigma$ and $R_g=20\varepsilon$, where $\varepsilon=-1,0,1$. If $\varepsilon=-1$ we assume that $H> 4$. Assume also that $\Sigma$ is Einstein. Then there holds
\[
8\pi^2\chi(\Sigma)
\geq\frac{1}{24}\left(12\varepsilon+
\frac{3}{4}H^2\right)^2|\Sigma|,
\]
with the equality holding if and only if $(M,g)$ is isometric to a geodesic ball in a  simply connected space form.
In particular, $\chi(\Sigma)>0$. 
\end{theorem}
\begin{proof}
It follows from Theorem \ref{main-thmB} that
\[
\left(12\varepsilon+\frac{3}{4}H^2\right)|\Sigma|\leq\int_{\Sigma}R_{\gamma}d\Sigma.
\]
Using H\"older inequality  we obtain
\[
\left(12\varepsilon+\frac{3}{4}H^2\right)^2|\Sigma|\leq\int_{\Sigma}R^2_{\gamma}d\Sigma.
\]
Now recall that the Chern-Gauss-Bonnet formula says that 
\[
8\pi^2\chi(\Sigma)= \frac{1}{4}\int_{\Sigma}|W_\gamma|^2d\Sigma
+\frac{1}{24}\int_{\Sigma}R^2_{\gamma}d\Sigma-\frac{1}{2}
\int_{\Sigma}|\stackrel{\circ}{\rm Ric}_{\gamma}|^2d\Sigma,
\] 
where $W$ is the Weyl tensor. Under our assumptions, this leads to  
\[
8\pi^2\chi(\Sigma)\geq \frac{1}{24}\int_{\Sigma}R^2_{\gamma}d\Sigma,
\]
and the result follows.
\end{proof}

Let us now recall a result proved in \cite[Corollary 3]{MT1} which, under suitable assumptions, characterizes certain $V$-static manifolds as the limiting cases of  geometric inequalities involving the geometry of the boundary.

\begin{theorem}\label{MT1rig}
	Let $(M,g)$ be an $n$-dimensional compact $V$-static manifold with a
	connected boundary $\Sigma$ and $R_g=\varepsilon n(n-1)$.
	\begin{itemize}
		\item[(i)] If $\varepsilon=0$, $(\Sigma,\gamma)$ is isometric to a geodesic sphere in $\mathbb{R}^n$
		and $M$ is spin if $n\geq8$, then ${\rm Ric}_g(\nu,\nu)$ is a non-positive constant along $\Sigma$,
		and ${\rm Ric}_g(\nu,\nu)=0$ if and only if $(M,g)$ is isometric to a geodesic
		ball in $\mathbb{R}^n$;
		\item[(ii)] If $n=3$, $\varepsilon=0$ and ${\rm Ric}_g(\nu,\nu)=0$ along $\Sigma$, then
		$(M,g)$ is isometric to a geodesic ball in $\mathbb{R}^3$.
		\item[(iii)]  If $n=3$, $\varepsilon=-1$ and $(\Sigma,\gamma)$ is isometric to a geodesic sphere in
		$\mathbb{H}^3$, then ${\rm Ric}_g(\nu,\nu)$ is a constant satisfying ${\rm Ric}_g(\nu,\nu)\leq -2$ along
		$\Sigma$, and ${\rm Ric}_g(\nu,\nu)=-2$ if and only if $(M,g)$ is isometric to a
		geodesic ball in $\mathbb{H}^3$.
	\end{itemize}
\end{theorem}

We extend this result in the following way; see Remark \ref{ultimar} below. 

\begin{theorem}\label{main-thmBBB}
	Let $(M,g)$ be an $n$-dimensional $V$-static manifold with  a connected boundary and $R_g=\varepsilon n(n-1)$. Assume that $(\Sigma,\gamma)$ is isometric to a totally umbilic hypersurface $(\Sigma_\varepsilon,\gamma_\varepsilon)$ in a simply connected space form of curvature $\varepsilon$. Then there holds 
	\[
	H_\varepsilon\geq H\,
	\]
	where $H$ and $H_\varepsilon$ are the mean curvature of $\Sigma$ and $\Sigma_\varepsilon$, respectively. Equality holds if and only if $(M,g)$ is isometric to a geodesic ball in the corresponding simply connected space form.
\end{theorem}

\begin{proof}
	It follows from  Gauss equation that
	\[
	2{\rm Ric}_g(\nu,\nu)+R_{\gamma}= \varepsilon n(n-1)+\frac{n-2}{n-1}H^2
	\]
	and
	\[
	R_{\gamma_\varepsilon}=\varepsilon(n-2)(n-1)+\frac{n-2}{n-1}H_{\varepsilon}^2\,.
	\]
	As $(\Sigma,\gamma)$ is isometric to $(\Sigma_\varepsilon,\gamma_\varepsilon)$, we have $R_{\gamma}=R_{\gamma_\varepsilon}$. Hence,
	\begin{equation}\label{ultima}
	2{\rm Ric}_g(\nu,\nu)-2\varepsilon (n-1)=\frac{n-2}{n-1}(H^2-H_{\varepsilon}^2)\,.
	\end{equation}
	By integrating over $(\Sigma,\gamma)=(\Sigma_\varepsilon,\gamma_\varepsilon)$
	and using (\ref{ineqvstat}), the inequality follows 
	since $H$ and $H_{\varepsilon}$ are positive constants.
	Finally, if 
	 $H=H_{\varepsilon}$ we conclude that  $(M,g)$ is Einstein and the rigidity statement  follows from Theorem \ref{MTam}.
\end{proof}

\begin{remark}
	\label{ultimar}{\rm Notice that (\ref{ultima}) justifies our claim above that Theorem \ref{main-thmBBB} extends Theorem \ref{MT1rig}.}
\end{remark}

\begin{corollary}
	Let $(M^n,g)$ compact $V$-static manifold with a connected boundary $\Sigma$ and  $R_g=n(n-1)$. Then $(\Sigma,\gamma)$ can not be isometric to the standard unit $(n-1)$-sphere.
\end{corollary}

\begin{proof}
	If the isometry exists then the theorem yields the contradiction $H=0$.
\end{proof}

\begin{remark}
	\label{remspin}{\rm If $\varepsilon=0,-1$ and $M$ is spin then  Theorem \ref{main-thmBBB} can be improved in the sense that it suffices to assume that $\gamma\geq\gamma_{\varepsilon}$; see Section \ref{vstatspin}.}
\end{remark}

\section{Applications of the generalized Pohozahev-Schoen identity to  generalized solitons and an Alexandrov-type theorem}
\label{applpsvstac}

So far we have presented applications of the Pohozaev-Schoen identity (\ref{schoenid2}). We now discuss a few applications of the  generalized Pohoz{a}ev-Schoen identity (\ref{schoenpohgen}). In particular, we present an alternative approach to results appearing in \cite{BBR} and \cite{GWX}. Even though those results are not directly related to static-type metrics, which is our main focus here, we include them in order to illustrate the flexibility of the method. Even in the ``classical'' case where $B=E_g$, the Einstein tensor, the  use of (\ref{schoenid2}) substantially simplifies the proofs of results otherwise obtained by involved computations. We also include here an Alexandrov-type theorem for hypersurfaces suggested by Example \ref{hyper}.  

We start with a definition which encompasses concepts introduced in \cite{PRRS} and \cite{GWX}.

\begin{definition}\label{gensol}
	An $h${\rm -generalized almost  soliton} is a complete Riemannian manifold $(M^{n},g)$ endowed with a locally conserved symmetric tensor $B$, a vector field $X$ and  two smooth functions $\mu$ and $h$
	satisfying 
	\[
	B+\frac{h}{2}\mathcal{L}_{X}g=\mu g.
	\]
	In case $h=1$ we say that $(M,g,B, X,\mu)$ is a {\rm generalized almost soliton}. If $X$ is a gradient field then we say that the corresponding soliton is a {\em gradient}. Moreover, if $X$ is a Killing field then we say that the soliton is trivial. 
\end{definition}

\begin{theorem}\label{gensolbbr}
	Let $(M,g,B,X,\mu)$ be a generalized almost soliton with $M$ closed.
	\begin{itemize}
		\item[(i)] If $b=g^{ij}B_{ij}$ is constant then $X$ is a conformal vector field. Moreover, if $X$ is a  gradient field then  $(M,g)$ is conformally equivalent to a round sphere.
		\item[(ii)] If $B=E_g$  and the almost Ricci soliton $(M,g)$ is non-trivial and has constant scalar curvature then it  is  a gradient and $(M,g)$ is isometric to a standard sphere. In fact, $X$ decomposes as a sum of a Killing and a gradient field.   
	\end{itemize}
\end{theorem}

\begin{proof}
	By definition we have that $\mathcal{L}_{X}g=-2\stackrel{\circ}{B}+fg$ for some $f$. Then using the generalized Pohoz{a}ev-Schoen identity in the form (\ref{schpohocol1}) we obtain that 
	$\int_{M}|\stackrel{\circ}{B}|^{2}dM=0$ 
	and hence $\stackrel{\circ}{B}\,=0$.
	It follows  that $\mathcal{L}_{X}g=fg$ and therefore $X$ is a conformal vector field. By setting $f=2\varphi$ and taking trace we get $\varphi={\rm div}_gX/n$. Thus, if $X$ is a gradient field, $X=\nabla_g\psi$, then $\psi$ satisfies 
	\[
	\nabla^2_g\psi=\varphi g=\frac{\Delta_g\psi}{n}g.
	\]
	That $(M,g)$ is conformally equivalent to a round sphere now follows from \cite[Theorem 6.3]{Y}.
	
	Now, if $B={E}_g$ then  the argument above based on (\ref{schpohocol1}) shows that $(M,g)$ is Einstein and $X$ is a conformal vector field, $\mathcal L_Xg=2\varphi g$. With this information at hand, the assertions in (ii) follow easily from well-known results. For instance, we may proceed as follows.
	First observe that, under these conditions, it is known that the conformal potential $\varphi$ satisfies the 
	identity
	\begin{equation}\label{yanoeq}
	\nabla^2_g\varphi=-\frac{R_g}{n(n-1)}\varphi g
	\end{equation}
	see the proof of \cite[Lemma 2.2]{H1}.
	Now assume that $\varphi$ is a constant function. From (\ref{yanoeq}) we see that either $R_g=0$ or $\varphi=0$. In the first case, it follows from 
	\cite[Theorem 6]{O} that $X$ is
	Killing, which contradicts the non-triviality of the soliton. Also, if $\varphi=0$ we have again that $X$ is Killing. Thus, $\varphi$ is a non-constant function so that $X$ is a non-homothetic conformal field and that $(M,g)$ is isometric to a round sphere now follows from \cite{NY}.   
	
	Finally, take trace in (\ref{yanoeq}) to get
	\[
	\Delta_g\varphi=-\frac{R_g}{n-1}\varphi,
	\] 
	that is, $R_g/(n-1)$ is a non-trivial eigenvalue of $\Delta_g$. In particular, $R_g>0$. 
	Now set
	\[
	u=-\frac{n(n-1)}{R_g}\varphi
	\]
	and check directly that 
	$\mathcal L_Xg=\mathcal L_{\nabla_g u}g$.
	This shows that $X-\nabla_gu$ is Killing.
\end{proof}

Item (ii) above corresponds to \cite[Corollary 1]{BBR}. By using  the same argument we can approach \cite[Theorem 1]{GWX}; see item $({\rm ii})$ in the next theorem.

\begin{theorem}\label{gensolgwx}
	Let $(M,g,B,X,\lambda,h)$ be a generalized $h$-almost soliton with $M$ closed. 
	\begin{itemize}
		\item[(i)] If  $v=g^{ij}B_{ij}$ is constant and $h$ has a fixed sign, i.e., $h>0$ or $h<0$, then $X$ is a conformal vector field. In particular, if $X$ is gradient field then  $(M,g)$ is conformally equivalent to a round sphere. 
		\item[(ii)] If $B=E_g$, and the generalized almost Ricci soliton $(M,g)$ is non-trivial, has constant scalar curvature and $h$ has a fixed sign then it is a gradient and and $(M,g)$ is isometric to a standard sphere. In fact, $X$ decomposes as a sum of a Killing and a gradient field.  
	\end{itemize}
\end{theorem}

Based on Example \ref{hyper} we now present an  Alexandrov-type theorem for  hypersurface 
$(M^n,g)$
immersed in an Einstein manifold $(\widetilde M^{n+1},\widetilde g)$ satisfying an identity of the type
\begin{equation}\label{alex}
II+\nabla_g^{2}f=\mu g,
\end{equation}
where $\mu$ and $f$ are smooth functions on $M$ and $II$ is the second fundamental form.

\begin{theorem}\label{alexth}
	Let $(M,g)$ be as above and assume that
	the mean curvature $H=g^{ij}II_{ij}$ is constant. Then $M$ is a totally umbilical hypersurface. In particular, if $\widetilde M$ is a simply connected space form then $M$ is a geodesic sphere.
\end{theorem}

\begin{proof}
	We first construct a locally conserved tensor on $M$. Let $\left\{e_{i}\right\}$, $1\leq i\leq n$ be a local orthonormal frame on $M$,  which we extend to $\widetilde M$ by setting $e_{n+1}=\nu$. Let $h_{ij}=II(e_{i},e_{j})$, $1\leq i,j\leq n$. It follows from Codazzi equations that
	$$\nabla^{k}h_{ij}-\nabla^{j}h_{ik}=\widetilde{R}_{n+1ijk},\quad 1\leq i,j,k\leq n,$$
	where $\widetilde{R}_{ijkl}$ is the curvature  tensor of $(\widetilde M,\widetilde{g})$. If we take $j=i$ in this identity and sum over $i$ we get
	$$
	\nabla^{k}H=\displaystyle\sum_{i=1}^{n}h_{ii,k}=\displaystyle\sum_{i=1}^{n}\nabla^{i}h_{ik}+\displaystyle\sum_{i=1}^{n}\widetilde{R}_{n+1iik}=\displaystyle\sum_{i=1}^{n}\nabla^{i}h_{ik}+\widetilde{\rm Ric}_{n+1k}.
	$$
	Since $\widetilde M$ is Einstein, $\widetilde{\rm Ric}_{n+1k}=\frac{\widetilde R}{n+1}\widetilde{g}_{n+1k}=0$ for each $k=1,\ldots, n$, and hence $\nabla^{i}II_{ij}=\nabla^{j}H$. Therefore, the tensor $B=II-Hg$ is  locally conserved, ${\rm tr}_gB=(1-n)H$ and ${\stackrel{\circ}{B}}={\stackrel{\circ}{II}}$. Thus, using (\ref{schpohocol1}) and (\ref{alex}) we obtain
	$$\int_{M}\langle \nabla_g H, \nabla_g f\rangle\,dM=\frac{n}{n-1}\int_{M}|{\stackrel{\circ}{II}}|^{2}dM,$$
	and the result follows.
\end{proof}

\section{$V$-static manifolds and spinors}\label{vstatspin}

In this section we show how Theorem \ref{main-thmBBB} can be improved in case $\varepsilon=0,-1$ if we assume that the $V$-static manifold $(M,g)$ is spin. Here we follow closely the ideas in \cite{HMZ} \cite{HMR} \cite{HMRa}.   
We retain the notation of Section \ref{vstatic}.

We set $\epsilon=\sqrt{\varepsilon}$ so that $\epsilon=0,i$ if $\varepsilon=0,-1$, respectively. 
We consider a compact spin manifold $M$ of dimension $n\geq 2$ endowed with a Riemannian metric $g$. Also, we assume that $M$ carries a non-empty boundary $\Sigma$. We fix a spin structure on $M$ and denote by $\mathbb SM$ the corresponding spin bundle. If $\nabla^g$ is the spin connection on $\mathbb SM$ induced by  $g$ we set
\[
\widetilde \nabla_X^g\psi=\nabla^g_X\psi+\frac{\epsilon}{2}X\cdot \psi,
\] 
where $X\in\Gamma(TM)$, $\psi\in\Gamma(\mathbb SM)$ and the dot means Clifford multiplication. The corresponding Dirac operator is 
\[
\widetilde D^g\psi=\sum_{i=1}^ne_i\cdot \widetilde \nabla^g_{e_i}\psi=D^g\psi-\frac{n\epsilon}{2}\psi,
\] 
where $D^g\psi=\sum_{i=1}^ne_i\cdot \nabla^g_{e_i}\psi$ is the standard Dirac operator. The integral Lichnerowicz formula reads
\begin{equation}\label{lichform}
\int_M\left(|\widetilde \nabla^g\psi|^2-|\widetilde D^g\psi|^2+\frac{R_g-\varepsilon n(n-1)}{4}|\psi|^2\right)dM=\int_{\Sigma}
\left\langle {\widetilde D}^\gamma\psi-\frac{H}{2}\psi,\psi\right\rangle d\Sigma,
\end{equation}
where  
\[
\widetilde D^\gamma \psi=D^\gamma\psi +\frac{n-1}{2}\epsilon\psi, 
\]
where $D^\gamma$ is the Dirac operator on $\mathbb S\Sigma$, the induced spin bundle on $(\Sigma,\gamma)$,  $\gamma=g|_\Sigma$. 

Starting from (\ref{lichform}), the following result is proved in \cite{HMZ} \cite{HMR}.

\begin{proposition}\label{propspin} \cite{HMZ} \cite{HMR}
	Let $(M,g)$ as above and assume that 
	$R\geq \varepsilon n(n-1)|\varepsilon|^2$ and $H\geq 0$.
	Then there holds
	\[
	\widetilde\lambda^\gamma\geq \frac{\underline H}{2},
	\]
	where  $\widetilde{\lambda}^\gamma$ be the lowest positive eigenvalue of $\widetilde D^\gamma$.
	Moreover, the equality implies that $(M,g)$ is Einstein (with $R_g=\varepsilon n(n-1)$) and $\Sigma$ has constant mean curvature.
\end{proposition}

\begin{corollary}\label{corspin}
	If $\lambda^\gamma$ is the lowest positive eigenvalue of 
	$ D^\gamma$ then
	\begin{equation}\label{spect1}
	{\lambda}^\gamma\geq \frac{1}{2}\sqrt{\underline H^2+\varepsilon (n-1)^2}.
	\end{equation}
	Moreover, the equality implies that $(M,g)$ is Einstein (with  $R_g=\varepsilon n(n-1)$) and $\Sigma$ has constant mean curvature.
\end{corollary}

\begin{proof}
	It is easy to check that  
	\[
	\widetilde{\lambda}^\gamma=\sqrt{(\lambda^\gamma)^2-\varepsilon\frac{(n-1)^2}{4}},
	\]
	which proves the result.
\end{proof}

Following \cite{HMRa}, we combine this with a Vafa-Witten-type bound due to Herzlich \cite{H2}.
Define $s_\varepsilon(r)=r$ if $\varepsilon=0$ and $s_\varepsilon(r)=\sinh r$ if $\varepsilon=-1$. Also, define $c_\varepsilon(r)=s_\varepsilon'(r)$. Thus, if $h_0$ is the round metric on the unit sphere $\mathbb S^{n-1}$ then $g_\varepsilon=dr^2+s^2_{\varepsilon}(r)h_0$ is the standard metric on the simply connected space form $\mathbb H_\varepsilon^n$ of curvature $\varepsilon$. 
Let $B_\varepsilon(r)\subset \mathbb H_\varepsilon^n$ be the ball of radius  $r$ and $S_\varepsilon(r)=\partial B_\varepsilon(r)$.  Note that $S_\varepsilon(r)$ has intrinsic metric $\gamma_\varepsilon(r)=s^2_\varepsilon(r)g_0$ and mean curvature $H_{\varepsilon}(r)=(n-1)f_\varepsilon(r)$, where $f_\varepsilon=c_\varepsilon/s_\varepsilon$.

Let $\gamma$ be a metric on $S_\varepsilon(r)$ satisfying $\gamma\geq \gamma_{\varepsilon}(r)$. A result by Herzlich \cite{H2} says that 
\begin{equation}\label{spect2}
\lambda^{\gamma}\leq \frac{n-1}{2s_\varepsilon(r)}.
\end{equation}
Moreover, the equality holds if and only if $\gamma=\gamma_\varepsilon(r)$. Combining (\ref{spect1}) with (\ref{spect2}) and usint that $\varepsilon s_\varepsilon^2+c_\varepsilon^2=1$, we get the following result.

\begin{theorem}\label{theo1}
	Let $(M,G)$
	be a Riemannian manifold satisfying $R\geq \varepsilon n(n-1)$ and $H\geq 0$ on $\Sigma$. Assume also that $\Sigma$ is diffeomorphic to $S_\varepsilon(r)$ and that $\gamma=g|_\Sigma\geq \gamma_\varepsilon(r)$ there. Then
	\[
	\underline H\leq H_{\varepsilon}(r). 
	\] 
	Moreover, the equality implies that $M$ is Einstein (with $R=\varepsilon n(n-1)$),  $\gamma=\gamma_\varepsilon(r)$ and $H=H_\varepsilon(r)$.
\end{theorem}

As an immediate consequence we obtain the following extension of Theorem \ref{main-thmBBB} when $\varepsilon=0,-1$ and $M$ is spin.

\begin{theorem}
	Let $(M,g)$ be a $V$-static spin manifold with $R_g=\varepsilon n(n-1)$, $\varepsilon=0,-1$. Assume that $\Sigma$ is diffeomorphic to $S_\varepsilon (r)$ and that $\gamma=g|_\Sigma\geq \gamma_\varepsilon(r)$ there. Then 
	\[H\leq H_\varepsilon(r).
	\] 
	Moreover, the equality implies that 
	$(M,g)=(B_\varepsilon(r),g_\varepsilon)$.
\end{theorem}  

\begin{proof}
	Just observe that the equality  $H= H_\varepsilon(r)$ implies that $(M,g)$ is Einstein and use Theorem \ref{MTam}.
\end{proof}

This corresponds to \cite[Theorem 2]{HMRa} in our setting.
Note also that for $n=3$ the spin condition is superfluous.

\end{document}